\documentclass{birkau}

\usepackage{amsmath,amssymb,url}
\usepackage{amssymb}
\usepackage [all]{xy}
\usepackage{mathrsfs}
\usepackage{mathtools}
\usepackage{stmaryrd}
\usepackage{lmodern}
\usepackage{faktor}

\usepackage{amssymb}
\usepackage[percent]{overpic}
\usepackage{enumerate}
\usepackage{breqn}

\numberwithin{equation}{section}

\theoremstyle{plain}
\newtheorem{theorem}{Theorem}[section]
\newtheorem{lemma}[theorem]{Lemma}

\theoremstyle{definition}
\newtheorem{definition}[theorem]{Definition}

\newtheorem{remark}[theorem]{Remark}
\newtheorem{example}[theorem]{Example}

\DeclareMathOperator{\Diagram}{Diagr}

\begin{document}


\title[A note on the probability of a groupoid having deficient sets]{A note on the probability of a groupoid \\ having deficient sets\footnote{Under review at \emph{Algebra Universalis.}} }

\author[C. Card\'o]{Carles Card\'o}
\address{Departament de Ci\`encies Bàsiques, \\Universitat Internacional de Catalunya, c/ Josep Trueta s/n,  \\Sant Cugat del Vallès,  08195, Spain. }
\email{ccardo@uic.es}





\begin{abstract} 
A subset $X$ of a groupoid is said to be deficient if $|X \cdot X|\leq |X|$. It is well-known that the probability that a random groupoid has a deficient $t$-element set with $t\geq 3$ is zero. However, as conjectured in \cite{Freese2022}, we show that the probability is not zero in the case of sets of two elements and calculate the exact value. We explore some generalisations on deficient sets and their likelihoods.
\end{abstract}

\maketitle


\section{Introduction}
 
A property of ``almost all'', or ``almost securely'', for finite algebras is a property that is satisfied with a probability tending to one when the size of the algebras goes infinite. There are many examples of such properties. For instance, almost all finite algebras are rigid, i.e., their automorphism group is trivial. Or, for instance, almost all finite algebras with at least two operation symbols, at least one non-unary, have no trivial subalgebras. See \cite[\S 6.2]{Bergman2012} or \cite[\S 9.7]{Freese2022}.

Although there are two types of probabilities to consider (one is based on counting labelled algebras, while the other is based on counting isomorphic classes of algebras), it turns out that, as Freese proved \cite{Freese1990}, both types coincide for groupoids and in general for algebras with a mild condition on the signature.

One of the most accurate known results of ``almost all'' is due to Murski\v{\i} in the 1970s. He proved that almost all finite algebras are \emph{idemprimal}, meaning that the algebra operations can express any idempotent term function. Despite the several characterisations of idemprimality, the proof relies on showing that a random algebra does not admit a \emph{reduced set} \cite[Definition~9.95]{Freese2022}. The proof, \cite[Lemma~9.100]{Freese2022} inspects ten cases, many involving the notion of a deficient set. A subset $X$ of a groupoid is deficient if $|X\cdot X | \leq |X|$. For example, the second case states that almost all groupoids have no deficient sets with more than two elements. In addition, in many other cases, the proof considers 2-element deficient sets with some extra condition that makes the case not admissible for random groupoids. In this respect, in \cite[p. 93]{Freese2022}, the authors comment:
\begin{quote}
Notice that most of the statements (i)-(x) assert the existence of a 2-element subset with some additional property. Unfortunately, it seems not to be the case that a random finite groupoid has a probability 0 of having a deficient 2-element set. Something approaching the complexity of the Lemma seems to be required of any proof of Murski\v{\i} Theorem.
\end{quote}
Intrigued by this remark, we have calculated the probability that a random groupoid has a deficient 2-element set. That and some generalisations on deficient sets constitute the theme of the present note.
  
Let us fix the notation. The classical combinatorial functions
${n \choose m}$, ${ n \brace m }$, $\big[ n \big]_m$, $n!\! !$, and $n!\! ! \! !\,$,
are called, respectively, the \emph{combinatorial number $n$ over $m$}, the \emph{Stirling number of second type $n$ over $m$}, the \emph{falling factorial of $n$ $m$ times},
the \emph{double factorial} and the \emph{triple factorial} of $n$. For definitions and properties, see, for example, \cite{knuth2003concrete}. Let us recall that ${n \brace m}$ counts the number of partitions of an $n$-element set into blocks of size $m$, and that $\big[n \big]_m= n(n-1) \cdots (n-m+1)$.
By a groupoid $A$ of order or size $n$, we mean a base set $A=[n]=\{1, \ldots, n\}$ and a binary operation.
Given a non-trivial groupoid $(A,\cdot)$, and $\{i,j\} \subseteq A$ with $i<j$, we will abbreviate the portion of the Cayley table on values $i,j$ by a matrix,
$$ \begin{pmatrix} x & y \\ z & t \end{pmatrix}=
\mbox{\begin{tabular}{ c | c c }
$\cdot$ & $i$ & $j$ \\
\hline
$i$ & $x$ & $y$ \\
$j$ & $z$ & $t$
\end{tabular}} $$

\begin{definition}
Given a groupoid $A$, and a 2-element subset $X \subseteq A$,
\begin{enumerate}[(i)]
\item  we say that a $X$ is of \emph{type} $T_0$ if its portion of the Cayley table is of the form
$$\begin{pmatrix} x & x \\ x & x \end{pmatrix},$$
for some $x \in A$;
\item  and we say that $X$ is of \emph{type} $T_1, \ldots, T_7$, if its portion of the Cayley table is one of the seven following forms, respectively,
$$\begin{pmatrix} x & y \\ y & y \end{pmatrix}, \begin{pmatrix} y & x \\ y & y \end{pmatrix}, \begin{pmatrix} y & y \\ x & y\end{pmatrix},\begin{pmatrix} y & y \\ y & x \end{pmatrix},  \begin{pmatrix} x & x \\ y & y \end{pmatrix},\begin{pmatrix} x & y \\ x & y\end{pmatrix},\begin{pmatrix} x & y \\ y & x \end{pmatrix}, $$
for some $x,y \in A$ with $x\not=y$.
\end{enumerate}
\end{definition}

Notice that any deficient 2-element set pertains to one, and only one, type from 0 to 7. By $X_{t,n}$, we denote the random variable that counts the number of deficient 2-subsets of the type $t=0, \ldots, 7$ in a random groupoid of size $n$. We define the probability
$\Pr(X_{t,n}=m)$ as the fraction given by the number of labelled groupoids of size $n$ containing exactly $m$ deficient sets of type $t$ divided by $n^{n^2}$. Then, we define $X_t$ as the random variable with the probability distribution
$$\Pr(X_{t}=m)=\lim_{n\to \infty} \Pr(X_{t,n}=m).$$
Finally, define $X=\sum_{t=0}^7 X_i$. We are interested in calculating $\Pr(X\geq 1)$. Let us make a pair of easy observations before tackling the problem.

\begin{enumerate}[(i)]

\item The probability that a groupoid $A$ of size $n\geq 2$ has a deficient set $\{i,j\}$ of type 0 is
$$ \frac{n^{n^2-4}\cdot n}{n^{n^2}}=\frac{1}{n^3},$$
because the type 0 is determined by only one parameter $x$, that is $i\cdot i =i \cdot j=j \cdot i= j\cdot j=x$, and then there are $n^{n^2-4}$ entries of the Cayley table which can take independently any value. Therefore,
$$\Pr(X_0 \geq 1)\leq \lim_{n\to \infty} {n\choose 2}\frac{1}{n^3}=0,$$
and $\Pr(X_0=0)=1$. Thus, $\Pr(X=m)=\Pr(X_1+\cdots +X_7=m)$ and we do not need to consider the type 0 for the calculations.

\item However, two parameters $x,y$ determine the other seven types.
The probability that a groupoid of size $n$ has a deficient set $\{i,j\}$ of type $ t \in \{1,\ldots,7\}$ is
$$ \frac{n(n-1)n^{n^2-4}}{n^{n^2}}=\frac{n-1}{n^3},$$
and then,
$$\Pr(X_t \geq 1)\leq \lim_{n\to \infty} {n\choose 2}\frac{n-1}{n^3}=\frac{1}{2}.$$
That bound does not ensure that the probability of having a deficient 2-element set is not null. To obtain the exact value, we need to develop it in more detail.

\end{enumerate}

\section{Calculation}

\begin{definition} A \emph{$k$-configuration of deficient 2-element sets} in a groupoid $A$ is a set of $k$ deficient 2-element sets of $A$ of types 1 to 7.

We will say that a configuration $C$ is \emph{disjoint} if $D\cap D' =\emptyset$ for each two different $D, D'\in C$. The \emph{disjoint sum} of two configurations $C, C'$ is denoted and defined as $C \oplus C' =C \cup C'$, provided that $D\cap D' =\emptyset$ for each $D \in C$ and $D'\in C'$.

The set of $k$-configurations in a groupoid of order $n$ is denoted by $\mathcal{C}_{k,n}$ and the subset of disjoint configurations is denoted by $\mathcal{D}_{k,n}$.
\end{definition}

\begin{figure}[tb] 
\centering
\begin{overpic}[scale=0.13]{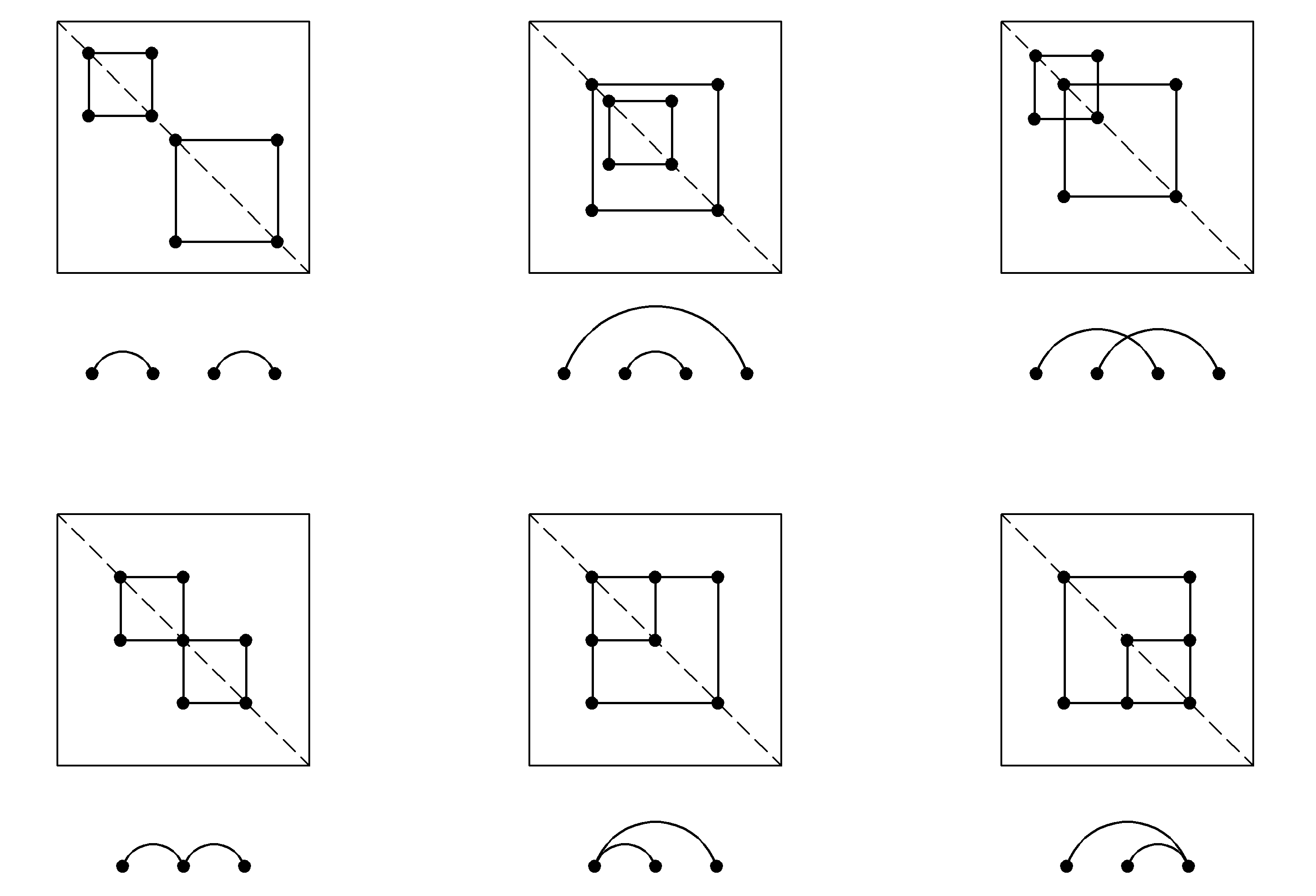}
\end{overpic}
\caption{Diagrams of all the 2-configurations and their schematic representation on the Cayley table. Labels on edges are not shown. The external square represents the Cayley table of the groupoid, and the squares inside represent the deficient 2-element sets. The first three are perfect matchings.}\label{Diagrams}
\end{figure}

\begin{definition}
A \emph{diagram} $G=(V,E, <, \ell)$ consists of
\begin{enumerate}[(i)]
\item a set of vertices $V$ with a total order $<$,
\item a symmetric graph $(V,E)$, $E \subseteq \{ \{v,w\} \mid v,w \in V\}$ such that $\deg(v)\geq 1$ for each vertex $v\in V$,
\item a mapping labelling the edges with a type $\ell:E\longrightarrow \{T_1,\ldots, T_7\}$ .\end{enumerate}
A diagram is a \emph{perfect matching} if $\deg(v)=1$ for each vertex $v$.
We say that two diagrams $G=(V,E,<, \ell)$ and $G'=(V',E',<', \ell')$ are equivalent if there is an isomorphism of graphs $f: V  \longrightarrow V'$, preserving the orders and the labellings.
\end{definition}

An alternative name for perfect matching is that of \emph{Brauer diagram} \cite{terada2001brauer}. For our purpose, we only need to know that a Brauer diagram, or perfect matching, has an even number of vertices $2k$, with $k$ edges, and there are $(2k-1)!\! !$ perfect matchings \cite{Callan2009}.

\begin{lemma} Given a groupoid $A$ and a $k$-configuration $C$, consider the 4-upla $\Diagram(C)=(\bigcup C, C, <, \ell)$ where $<$ is the natural order of integers and $\ell(D)$ is the type corresponding to the deficient set $D$. Then, $\Diagram(G)$ is a diagram. In addition, $C$ is disjoint if and only if $\Diagram(G)$ is a perfect matching.
\end{lemma}
\begin{proof} It is obvious that $\Diagram(G)$ has no isolated vertices because each vertex belongs to some deficient set, so the degree of each vertex is never zero. For the second statement, if $C$ is disjoint, $D \cap D'=\emptyset$ for any $D, D'\in C$ with $D\not=D'$. Therefore, $D$ and $D'$, understood as edges of the graph, do not share any vertex, which means that all the vertices have degree 1. The converse is trivial.
\end{proof}

\begin{definition} We say that two configurations are equivalent, denoted by $C\sim C'$, if $\Diagram(C)$ and $\Diagram(C')$ are equivalent. We denote by $\widetilde{C}$ the equivalence class of $C$.
\end{definition}

\begin{example}
Figure~\ref{Diagrams} shows schemes of 2-configurations and their diagrams without the labellings. Considering the labellings, we have $7\cdot 7\cdot 6=294$ possible diagrams. 
\end{example}

\begin{remark} \label{Remark1} All the two-edged diagrams are realisable as a configuration diagram. However, consider the diagram $G=(V,E, <,\ell)$, where $V=\{a,b,c\}$,  $a<b<c$, $E=\{ \{ab\}, \{b,c\}, \{a,c \}\}$ and $\ell(\{a,b\})=T_1$,   $\ell(\{a,b\})=T_4$, $\ell(\{a,b\})=T_7$. $G$ is not the diagram of any configuration. Hence, there are more diagrams than equivalence classes of configurations. 
\end{remark}

\begin{lemma} We have that
$$\left|\faktor{\mathcal{C}_{k,n}}{\sim}\right| \leq 7^k \cdot 2^{2k^2-k}, \qquad \left|\faktor{\mathcal{D}_{k,n}}{\sim}\right|=\begin{cases} 7^k (2k-1)!\! ! & \mbox{ if } n\geq 2k, \\ 0 & \mbox{ otherwise.} \end{cases}$$
\end{lemma}
\begin{proof} There is an injection from the set of equivalence classes of configurations into the set of equivalence classes of diagrams. Recall that in Remark~\ref{Remark1}, we commented that there are more diagrams than configurations, but that suffices to get the bound. Each edge has two vertices, so the graph has at most $2k$ vertices, and each subset of edges defines a graph. Therefore, there are at most $2^{2k \choose 2}$ possible graphs and $7^k$ possible labellings for each graph. 

 Regarding the second statement, there is a bijection between the set of equivalence classes of disjoint configurations and the set of equivalence classes of perfect matchings. There are $(2k-1)!\! !$ perfect matchings. Edges of each graph can be decorated in $7^k$ ways. 
\end{proof}

\begin{definition} Given a set $X \subseteq A$ of a groupoid $(A,\cdot)$, we denote by $X^*$ the operation of the partial groupoid 
 $$X^* =\{ (x,y,x\cdot y) \mid x,y \in X\} \subseteq \cdot: A\times A \longrightarrow A.$$  
\end{definition}

\begin{lemma} \label{Lemma3} Let $C$ be a configuration and set the quantities:
$$\alpha(C)=\left|\left(\bigcup C\right)\cdot \left(\bigcup C\right)\right|, \qquad\beta(C)=\left|\bigcup_{D \in C} D^*\right|, \qquad \gamma(C)=\left|\bigcup C\right|.$$
Given a diagram $G=(V,E,<,\ell)$, we denote by $v(G)$ the number of vertices, $k(G)$ the number of edges and $c(G)$ the number of connected components of the base graph. If $G=\Diagram(C)$ for some $k$-configuration, then $k=k(G)$ and
\begin{enumerate}[(i)]
\item $\alpha(C)\leq k(G)+c(G)$;
\item If $C$ is disjoint, $\alpha(C)=2k(G)$;
\item $\beta(C)=2k(G)+v(G)$;
\item $\gamma(C)=v(G)$;
\item $c(G)\leq k(G)$, and the equality holds iff $C$ is disjoint.
\end{enumerate} 
\end{lemma}
\begin{remark} We will need the above functions to count groupoids having certain configurations. The elements in $\bigcup C$ define a partial groupoid. The quantity $\alpha(C)$ is the number of necessary parameters to describe the partial groupoid. The quantity $\beta(C)$ is the number of entries of the partial groupoid  $\bigcup C$ in the Cayley table. The quantity $\gamma(C)$ is transparent: it just counts the number of elements of the partial groupoid.  
\end{remark}
\begin{proof}[Proof of Lemma~\ref{Lemma3}] The statement $k=k(G)$ is clear. 
\begin{enumerate}[(i)]
\item It is easily seen that $\alpha(C\oplus C')=\alpha(C) + \alpha(C')$. Let $G=\Diagram(C)$. We will prove that for a connected graph $G$, $\alpha(C)\leq k(G)+1$. Therefore, if $G$ has $c(G)$ connected components, $\alpha(C)\leq k(G)+c(G)$. 
Suppose that $G$ is connected. Consider first the case that $G$ is a path $G=P$. By induction on the number of vertices, it is easy to check that $\alpha(P)=v(P)$. Next, by induction on the number of edges, suppose that $G'=\Diagram(C')$ is connected but not a path. Then, there is at least an edge $D=\{i,j\}$ such that if we remove it, the resulting graph $G=\Diagram(C)$ is still connected. When we put the edge again, there are two possibilities. First, $i\cdot i= j \cdot j$. Then, $D$ is one of the types $T_2$, $T_3$ or $T_7$, we are adding a new parameter, and $\alpha(C') =\alpha(C)+1$. Second, $i \cdot i \not= j \cdot j$. Then, $D$ is one of the types $T_1, T_4, T_5$, or $T_6$, no new parameter is added and $\alpha(C')=\alpha(C)$.  In any of both cases, $\alpha(C') \leq \alpha(C)+1$. Since any connected diagram is obtained by adding edges to a path diagram $P$ with $v(G)$ vertices, and a path diagram has $k(P)=v(G)-1$ edges, we have to add $k(G)-k(P)=k(G)-v(G)+1$ edges. Since $\alpha(P)=v(P)=v(G)$, after to add the edges we get:
\begin{multline*}
\alpha(G) \leq v(G)+\underbrace{(1+ \cdots+ 1)}_{k(G)-v(G)+1 \mbox{ \footnotesize times}} =v(G)+ (k(G)-v(G)+1)=k(G)+1.
\end{multline*}

\item If $C$ is disjoint, it is conformed by $c(G)$ disjoint paths of legth 2. As we have viewed in point (i), the number of parameters $\alpha$ for a path with two vertices is two. Since we have $k(G)$ connected components $\alpha(C)=2k(G)$.

\item We follow a similar scheme to (i). It is easily seen that $\beta(C\oplus C')=\beta(C) + \beta(C')$. We will prove that for a connected graph $G$, $\beta(C)=2k(G)+v(G)$, and then the result comes from the additivity of $k(G)$ and $v(G)$ for disjoint graphs. 
Let $C$ be a configuration such that $\Diagram(C)$ is a path $P$. It can be proved by induction on the number of vertices that $\beta(C)=3v(G)-2$. By induction on the number of edges, suppose that $G'=\Diagram(C')$ is connected but not a path. Then, there is at least an edge $D$ such that if we remove it, the resulting graph $G=\Diagram(C)$ is still connected. Then, when we put the edge again, we have that $\beta(G')=\beta(G)+2$. Since any connected diagram is obtained by adding edges to a path diagram $P$ with $v(G)$ vertices, and a path diagram has $k(P)=v(G)-1$ edges, we have to add $k(G)-k(P)=k(G)-v(G)+1$ edges. Since $\beta(P)=3v(P)-2=3v(G)-2$, after to add the edges we get:
\begin{multline*}
\beta(G) = 3v(G)-2+\underbrace{(2+ \cdots+ 2)}_{k(G)-v(G)+1 \mbox{ \footnotesize times}} =3v(G)-2+ 2(k(G)-v(G)+1) \\=2k(G)+v(G).
\end{multline*}

\item Trivial. 
\item Since the degree of a vertex is at least one, there are no isolated vertices, whereby each connected component of the graf contains at least an edge. Hence, $c(G)\leq k(G)$.  We have equality for perfect matchings because each connected component has one edge. Suppose $c(G)= k(G)$ and let $G'$ a connected component of $G$ such that $1=c(G')<k(G')$. Necessarily, there is another component $G''$ such that $1=c(G'')>k(G'')=0$, meaning that $G''$ is an isolated vertex, which is absurd.
\end{enumerate} 
\end{proof}

\begin{theorem} \label{TheoremProb} The probability of a random groupoid has a deficient 2-element set is
$$1-\frac{1}{\sqrt{e^7}}=0.9698\ldots$$
\end{theorem}
\begin{proof} Denote by $\mathcal{G}^{n}_D$ the set of groupoids which have the 2-element deficient set $D \subseteq [n]$. By the inclusion-exclusion principle, the probability of a $n$-element groupoid having at least one 2-element deficient set is
$$\Pr(X\geq 1)=\lim_{n\to\infty} \frac{1}{n^{n^2}} \left| \bigcup_{ D\subseteq [n] } \mathcal{G}^{n}_D \right|=\lim_{n\to\infty} \frac{1}{n^{n^2}} \sum_{k=1}^{{n\choose 2}} (-1)^{k+1} \left| \bigcap_{D_1,\ldots, D_k\subseteq [n]} \mathcal{G}^{n}_{D_i} \right|.$$
We must introduce the limit twice in that expression to calculate that probability. First, we use the Bonferroni inequalities. Let $K$ an integer such that $1<2K \leq {n \choose 2}$. Then
\begin{align*}
 \sum_{k=1}^{2K-1} (-1)^{k+1} \left| \bigcap_{D_1,\ldots, D_k\subseteq [n]} \mathcal{G}^{n}_{D_i} \right| \leq  \left| \bigcup_{ D\subseteq [n] } \mathcal{G}^{n}_D \right| \leq \sum_{k=1}^{2K} (-1)^{k+1} \left| \bigcap_{D_1,\ldots, D_k\subseteq [n]} \mathcal{G}^{n}_{D_i} \right|.
 \end{align*}
We can introduce the limit inside the sums because they are finite. That means if the limits exist, then, multiplying by $1/n^{n^2}$,
$$\Pr(X\geq 1)= \sum_{k=1}^{\infty} (-1)^{k+1} \lim_{n\to\infty} \frac{1}{n^{n^2}} \left| \bigcap_{D_1,\ldots, D_k\subseteq [n]} \mathcal{G}^{n}_{D_i} \right|.$$
Let us calculate that limit.
\begin{align*}
&\lim_{n\to\infty} \frac{1}{n^{n^2}} \left| \bigcap_{D_1,\ldots, D_k\subseteq [n]} \mathcal{G}^{n}_{D_i} \right| \\
=& \lim_{n\to\infty} \frac{1}{n^{n^2}} \sum_{ C \in \displaystyle  \mathcal{C}_{k,n}} \left( n^{\alpha(G)}  + O(n^{\alpha(G)-1})\right)  n^{n^2 - \beta(G)} \\ =&\lim_{n\to\infty} \frac{1}{n^{n^2}} \sum_{ \widetilde{C} \in \displaystyle \faktor{\mathcal{C}_{k,n}}{\sim} } \left( n^{\alpha(G)}  + O(n^{\alpha(G)-1})\right) n^{n^2-\beta(G)} {n \choose \gamma(C)}.
\end{align*}
In the first equality, we have used the number of groupoids with $k$ 2-element deficient sets given by the possible entries in its Cayley table. All the entries are independent outside the configuration $C$. Therefore, there are $n^{n^2 - \beta(C)}$ possible choices. Next, we need to count the possible entries inside the configuration, considering that some entries are equal. That means counting the number of the parameters defining $C$, that is, $\alpha(C)=\left| \left(\bigcup C \right)\cdot \left(\bigcup C \right)\right|$. Therefore, there are $n^{\alpha(C)}+O(n^{\alpha(C)-1})$ possible choices for the parameters.  
In the second equality, we have used that the values $\alpha(C), \beta(C)$ are invariants if the configurations are equivalent. That is a trivial consequence of the Lemma~\ref{Lemma3} since $\alpha(C)$ and $\beta(C)$ depend at the end only on the graph. Since $\gamma(C)=\left| \bigcup C \right|$, there are ${ n \choose \gamma(C)}$ possible subsets where we can put the configuration.

Let $\overline{\mathcal{C}}_{k,n}=\mathcal{C}_{k,n}\setminus \mathcal{D}_{k,n} $, and take the following decomposition, where $\sqcup$ means a disjoint union:
$$\faktor{\mathcal{C}_{k,n}}{\sim}=\faktor{\left(\overline{\mathcal{C}}_{k,n} \sqcup \mathcal{D}_{k,n}\right)}{\sim}=\left(\faktor{\overline{\mathcal{C}}_{k,n}}{\sim}\right)  \sqcup \left( \faktor{\mathcal{D}_{k,n}}{\sim} \right).$$
We can decompose the above sum into two sums running over the sets $\faktor{\overline{\mathcal{C}}_{k,n}}{\sim}$ and $\faktor{\mathcal{D}_{k,n}}{\sim}$. 
Calculate the limit of the first sum. The number of non-equivalent configurations is bounded by $7^k 2^{2k^2-k}$ and therefore is finite, and we can swap the limit and the sum, and using the inequality  $\alpha(C) \leq k + c(G)$ from Lemma~\ref{Lemma3}(i), we get:
\begin{align*}&\lim_{n\to\infty} \frac{1}{n^{n^2}} \sum_{ \widetilde{C} \in \displaystyle \faktor{\overline{\mathcal{C}}_{k,n}}{\sim} } \left( n^{\alpha(G)}  + O(n^{\alpha(G)-1})\right) n^{n^2-\beta(G)} {n \choose \gamma(C)}\\
\leq&\lim_{n\to \infty} \sum_{ \substack{  G=\Diagram(C)  \\ \widetilde{C} \in \displaystyle \faktor{\overline{\mathcal{C}}_{k,n}}{\sim}}} \frac{n^{k+c(G)}}{ n^{ 2k+v(G)}} {n \choose v(G)} = \sum_{ \substack{  G=\Diagram(C)  \\ \widetilde{C} \in \displaystyle \faktor{\overline{\mathcal{C}}_{k,n}}{\sim}}} \lim_{n\to \infty} \frac{n^{k+c(G)}}{ n^{ 2k+v(G)}} {n \choose v(G)} \\
=& \sum_{ \substack{  G=\Diagram(C)  \\ \widetilde{C} \in \displaystyle \faktor{\overline{\mathcal{C}}_{k,n}}{\sim}}} \lim_{n\to \infty} n^{c(G)-k} =0.
\end{align*}
Where we have used that $c(G)<k$ when a configuration is not disjoint, Lemma~\ref{Lemma3}(v). 
For the second sum, we have that $c(G)=k$ and $v(G)=2k$, Lemma~\ref{Lemma3}  (ii) and (v). Therefore,
\begin{align*}&\lim_{n\to\infty} \frac{1}{n^{n^2}} \sum_{ \widetilde{C} \in \displaystyle \faktor{\mathcal{D}_{k,n}}{\sim} } \left( n^{\alpha(G)}  + O(n^{\alpha(G)-1})\right)  n^{n^2-\beta(G)} {n \choose \gamma(C)}\\
=&\lim_{n\to \infty} \sum_{ \substack{  G=\Diagram(C)  \\ \widetilde{C} \in \displaystyle \faktor{\mathcal{D}_{k,n}}{\sim}}} \frac{\left( n^{2k}  + O(n^{2k-1})\right) }{ n^{ 2k+v(G)}} {n \choose v(G)}  
\\=& \lim_{n\to \infty}  \left| \faktor{\mathcal{D}_{k,n}}{\sim}\right|  \frac{\left( n^{2k}  + O(n^{2k-1})\right)}{n^{ 4k}} {n \choose 2k} \\
 = &\: 7^k (2k-1)!\! ! \lim_{n\to \infty}\frac{n^{2k}}{n^{ 4k}} {n \choose 2k} 
= 7^k (2k-1)!\! !  \frac{1}{(2k)!}=\left(\frac{7}{2}\right)^k \frac{1}{k!}. 
\end{align*}
Now we can resume the main calculation,
$$\Pr ( X \geq 1 ) = \sum_{k=1}^{\infty} (-1)^{k+1}\left(\frac{7}{2}\right)^k \frac{1}{k!}=1-\sum_{k=0}^{\infty} \left(-\frac{7}{2}\right)^k \frac{1}{k!}=1-e^{-7/2}. \qedhere$$
\end{proof}

Remaking the last proof now it is easy to prove that $\Pr(X_i \geq 1)= 1-1/\sqrt{e}=0.3934\ldots$, for any $1\leq i \leq 7$, where we recall that $X_i$ counts the number of deficient sets of type $i$ in a random groupoid. In addition, all the variables $X_1, \ldots, X_7$ are mutually independent, which can be proved using the Cauchy product of series.

\section{Some generalizations}

A first generalisation considers deficient sets for an operation $f$ of arity $d>2$, that is, a set $X$, such that $|f(X,\ldots, X)|\leq |X|$. It is easy to prove that given an integer $s\geq 3$, those random algebras do not admit a deficient s-element set. However, the probability that a $d$-ary operation has a 2-element deficient set is
$$1- \exp\left(-\frac{1}{2} { \,\,2^d \brace 2}\right)=1- \exp \left( \frac{1- 2^{2^d-1}}{2}\right).$$
We do not show the proof, but it is essentially similar to that of Theorem~\ref{TheoremProb}. We must consider that there are ${ \,\,2^d \brace 2}$ different types of partial algebras for a 2-element deficiency sets. Each type of deficient set is defined by partitioning the $2^d$ entries of the Cayley table into two sets. Recall that Stirling numbers count the number of partitions.

Another generalisation is possible by relaxing the notion of a deficiency. Given a groupoid, we say that a subset $X \subset A$ has \emph{exceedance} $\varepsilon$ if
$$|X\cdot X| = |X| + \varepsilon.$$
When $\varepsilon \leq 0$, $X$ is a deficcient set. Although the probability of a groupoid having deficient 3-element sets is null, by admitting some exceedency (less deficiency), it is not. First, show that the probability of a groupoid having a 3-element set $X$ with excedeency $\varepsilon \leq 2$ is null. Since $|X \cdot X|\leq 3+2=5$, to count types of sets, we have to consider partitions of the set of the nine entries of the Cayley table into five or fewer blocks, where each block represents a parameter. Thus, there are
$$\sum_{i=1}^5 { 9 \brace i}=18002$$
types of those sets. Then, the probability of a random groupoid with size $n$ having some one of these sets is bounded from above by
$$18002 \frac{1}{n^{9} }\big[ n \big]_5 { n \choose 3}  $$
which tends to 0 when $n$ goes to infinity. Now consider the probability that a random groupoid has a subset with exceedance three, for which we must generalise all the concepts of the previous pages. Six parameters define a matrix representing the partial groupoid whereby there are
$${ 9 \brace 6}=2646$$
types of 3-element sets with exceedance three. Configurations are defined similarly. To handle these configurations, we also consider diagrams, but instead of a base graph, we need 3-uniform hypergraphs \cite{Ouvrad}, where the hyperedges are given by the 3-element sets with exceedance three, and where we label each hyperedge with one of the 2646 types. Two configurations are equivalent if their diagrams are. Denote by $\mathcal{C}_{k,n}^3$ the set of configurations and $\mathcal{D}_{k,n}^3$ the number of disjoint configurations. It holds that
$$\left| \faktor{\mathcal{C}_{k,n}^3}{\sim} \right| \leq 2646^k 2^{3k \choose 3},$$
and we have the recurrence
$$ \left| \faktor{\mathcal{D}_{k,n}^3}{\sim} \right|= 2646 \frac{ (3k-1)(3k-2)}{2} \left| \faktor{\mathcal{D}_{k-1,n}^3}{\sim} \right|,$$
when $n\geq 3k$, which has the closed form
$$\left| \faktor{\mathcal{D}_{k,n}^3}{\sim} \right|= 2646^k \frac{1}{2^k} \frac{(3(k-1))!}{ (3(k-1))!\!!\!!}.$$
As in the proof of Theorem~\ref{TheoremProb} we use the inclusion-exclusion principle, and we separate the principal sum into two sums: one for equivalence classes of non-disjoint configurations and the other for disjoint. We can conjecture that Lemma~\ref{Lemma3} generalises in the form that there are analogous functions $\alpha, \beta, \gamma$ in terms of hypergraphs such that the first sum tends to zero as $n$ goes to infinity.
Thus, the probability of having a 3-element set with exceedance at most three is given by the alternated sum
\begin{align*}
&\sum_{k=1}^\infty (-1)^{k+1}\left(\frac{2646}{2}\right)^k\frac{(3(k-1))!}{ (3(k-1))!\!!\!!} \frac{1}{(3k)!} =\sum_{k=1}^\infty (-1)^{k+1} \left(\frac{2646}{6}\right)^k\frac{1}{ k!},\end{align*}
where we have used that $(3(k-1))!\!!\!!=3^{k-1} (k-1)!$. That yields the value
$$1-\frac{1}{e^{441}}.$$
Generalising calculations, we find that the probability of a groupoid having a $s$-element subset with exceedance $\varepsilon \geq s^2-2s$, with $s>1$, is not null.

\section{Declarations}

\begin{itemize}

\item[] \textbf{Ethical Approval}. Not applicable. 
\item[] \textbf{Competing interests}. No conflict of interest exists.
\item[] \textbf{Authors' contributions}. Not applicable. 
\item[] \textbf{Funding}. Not applicable. 
\item[] \textbf{Availability of data and materials}. Not applicable.

\end{itemize}






\end{document}